\documentclass{article}
\usepackage{amssymb,amsmath,amsthm}
\usepackage{graphicx}
\usepackage{hyperref}
\newcommand{\ie}{\emph{i.e.}}
\newcommand{\eg}{\emph{e.g.}}
\newcommand{\cf}{\emph{cf}}
\newcommand{\Nat}{\mathbb{N}}
\newcommand{\Real}{\mathbb{R}}

\newcommand{\sii}{L^2}
\newcommand{\Dom}{D}

\newcommand{\sgn}{\mathop{\mathrm{sgn}}\nolimits}
\newcommand{\eps}{\varepsilon}

\newtheorem{Theorem}{Theorem}
\newtheorem{Proposition}{Proposition}
\newtheorem{Corollary}{Corollary}
\theoremstyle{remark}
\newtheorem{Remark}{Remark}
\begin{document}
%
\title{\textbf{\Large
The asymptotic behaviour of the heat equation
in a twisted Dirichlet-Neumann waveguide
}}
\author{
David Krej\v{c}i\v{r}{\'\i}k$\,^a$ \ and \ Enrique Zuazua$\,^b$
}
\date{\small
\emph{
$a)$ Department of Theoretical Physics,
Nuclear Physics Institute,
A\-cad\-e\-my of Sciences, 25068 \v Re\v z, Czech Republic;
krejcirik@ujf.cas.cz
\medskip \\
$b)$ IKERBASQUE Research Professor,
BCAM - Basque Center for Applied Mathematics,
Bizkaia Technology Park, Building 500,
E-48160 Derio - Basque Country - Spain;
zuazua@bcamath.org
}
\bigskip \\
14 June 2010
}
\maketitle
\begin{abstract}
\noindent
We consider the heat equation in a straight strip,
subject to a combination of Dirichlet and Neumann boundary conditions.
We show that a switch of the respective boundary conditions
leads to an improvement of the decay rate of the heat semigroup of
the order of $t^{-1/2}$.
The proof employs similarity variables that lead to a non-autonomous
parabolic equation in a thin strip contracting to the real line,
that can be analyzed on weighted Sobolev spaces
in which the operators under consideration have discrete spectra.
A careful analysis of its asymptotic behaviour shows that
an added Dirichlet boundary condition emerges asymptotically
at the switching point,
breaking the real line in two half-lines,
which leads asymptotically to the $1/2$ gain on the spectral lower bound,
and the $t^{-1/2}$ gain on the decay rate in the original physical variables.

This result is an adaptation to the case of strips
with twisted boundary conditions
of previous results by the authors
on geometrically twisted Dirichlet tubes.
%
%
\end{abstract}
%
%
\newpage
\section{Introduction}
%
We consider the heat equation
\begin{equation}\label{heat.intro}
  u_t - \Delta u = 0
\end{equation}
in an infinite planar strip $\Omega:=\Real\times(-a,a)$ of half-width $a>0$,
subject to
\begin{equation*}
\left\{
\begin{aligned}
\mbox{Dirichlet boundary conditions on} \ \,
&
  \Gamma_\pi^D :=
  (-\infty,0)\times\{-a\}
  \cup
  (0,+\infty)\times\{a\}
\,,
\\
\mbox{Neumann boundary conditions on} \ \,
&
  \Gamma_\pi^N :=
  (0,+\infty)\times\{-a\}
  \cup
  (-\infty,0)\times\{a\}
\,,
\end{aligned}
\right.
\end{equation*}
and to the initial condition
\begin{equation}\label{heat.initial}
  u(\cdot,0) = u_0 \in \sii(\Omega)
  \,.
\end{equation}

This model is considered as a `twisted' counterpart
of the explicitly solvable problem given by (see~Figure~\ref{Fig.DN}):
\begin{equation*}
\left\{
\begin{aligned}
\mbox{Dirichlet boundary conditions on} \ \,
&
  \Gamma_0^D :=
  (-\infty,+\infty)\times\{-a\}
\,,
\\
\mbox{Neumann boundary conditions on} \ \,
&
  \Gamma_0^N :=
  (-\infty,+\infty)\times\{a\}
\,.
\end{aligned}
\right.
\end{equation*}

Henceforth we shall use the common subscript
$$
  \theta\in\{0,\pi\}
$$
when we want to deal with the two problems simultaneously
(the value of~$\theta$ suggests the rotation angle
giving rise to twisting/untwisting).

\begin{figure}[h!]
\begin{center}
\includegraphics[width=0.46\textwidth]{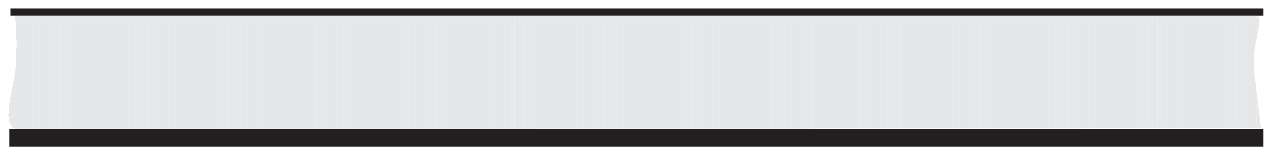}
\quad
\includegraphics[width=0.46\textwidth]{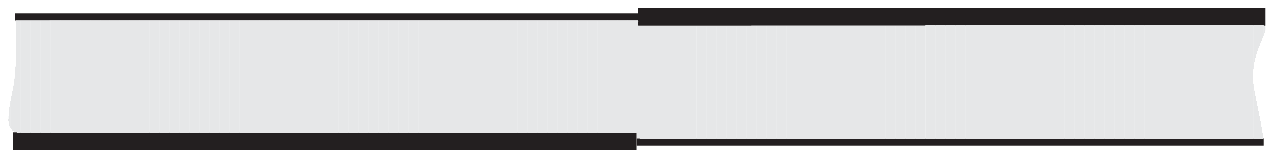}
\end{center}
\caption{Planar strips with untwisted (left)
and twisted (right) boundary conditions;
the thick and thin lines correspond to
Dirichlet and Neumann boundary conditions, respectively.}
\label{Fig.DN}
\end{figure}

The solution to~\eqref{heat.intro}--\eqref{heat.initial}
is given by $u(t) = e^{\Delta_\theta t} u_0$,
where $e^{\Delta_\theta t}$
is the semigroup operator on $\sii(\Omega)$
associated with the Laplacian~$-\Delta_\theta$
determined by the respective boundary conditions
(depending on~$\theta$).

The operators $-\Delta_{\pi}$ and $-\Delta_{0}$ have the same spectrum
\begin{equation}\label{spectrum}
  \sigma(-\Delta_{\theta}) =
  \sigma_\mathrm{ess}(-\Delta_{\theta}) =
  [E_1,\infty)
  \,, \qquad \mbox{where} \qquad
  E_1 := \left(\frac{\pi}{4a}\right)^2
  \,.
\end{equation}
Consequently, for all $t \geq 0$,
\begin{equation}\label{spectral.map}
  \big\|e^{\Delta_\theta t}\big\|_{\sii(\Omega)\to\sii(\Omega)}
  = e^{-E_1 t}
  \,,
\end{equation}
irrespectively of the value of~$\theta$.

In this paper, we are interested in additional time decay properties
of the heat semigroup, when the initial data are restricted
to a subspace of the Hilbert space~$\sii(\Omega)$.
We restrict ourselves to the weighted space
\begin{equation}\label{weight}
  \sii(\Omega,K)
  \qquad\mbox{with}\qquad
  K(x) := e^{x_1^2/4}
  \,,
\end{equation}
which means that the initial data are required
to be sufficiently rapidly decaying at the infinity of the strip.
As a measure of the additional decay, we consider
the (polynomial) \emph{decay rate}
\begin{equation}\label{rate}
  \gamma_\theta
  := \sup \Big\{ \gamma \left| \
  \exists C_\gamma > 0, \, \forall t \geq 0, \
  \big\|e^{(\Delta_\theta+E_1)t}\big\|_{
  \sii(\Omega,K)
  \to
  \sii(\Omega)
  }
  \leq C_\gamma \, (1+t)^{-\gamma}
  \Big\} \right.
  .
\end{equation}
Our main result reads as follows:
\begin{Theorem}\label{Thm.rate}
We have $\gamma_0=1/4$, while $\gamma_\pi \geq 3/4$.
\end{Theorem}

Here the power~$1/4$ corresponding to the untwisted case $\theta=0$
reflects the quasi-one-dimensional
nature of our model (recall that $d/4$ is the analogous decay rate
for the heat semigroup in the $L^2$-space
over the whole Euclidean space~$\Real^d$
for initial data in
$
  L^1(\Real^d) \supset \sii(\Real^d,e^{|x|^2/4} dx)
$).
Indeed, the result for~$\theta=0$ follows easily
by separation of variables (\cf~Section~\ref{Sec.straight}).

The essential content of Theorem~\ref{Thm.rate}
is that solutions to~\eqref{heat.intro}
when the strip is twisted (\ie~$\theta=\pi$) gain a further decay rate $1/2$.
The proof of this statement is more involved and constitutes the main body
of the paper (\cf~Section~\ref{Sec.ss}).
It is based on the method of self-similar solutions
developed in the whole Euclidean space
by Escobedo and Kavian~\cite{Escobedo-Kavian_1987}
and adapted to waveguide systems by the present authors in~\cite{KZ},
where it was shown that the heat kernel decays faster
in geometrically twisted tubes than in untwisted ones.
An open problem is to show that the decay rate~$\gamma_\pi$
is precisely~$3/4$ (\cf~Section~\ref{Sec.end}).

The way how to understand the difference in the decay rates
of Theorem~\ref{Thm.rate} is due to a fine difference between
the operators $-\Delta_0$ and $-\Delta_\pi$ in the spectral setting:
Although the operators have the same spectrum (\cf~\eqref{spectrum}),
%
%
the shifted operator $-\Delta_{0} - E_1$ is critical,
while $-\Delta_{\pi} - E_1$ is subcritical.
The latter is reflected in the existence of a Hardy-type inequality
\begin{equation}\label{Hardy}
  -\Delta_\pi - E_1 \geq \rho
\end{equation}
with a positive function~$\rho$
(while such an inequality cannot hold for $-\Delta_0 - E_1$).

Various Hardy inequalities for $-\Delta_{\pi} - E_1$
were established in~\cite{KK2}.
A general conjecture on the influence
of the subcriticality of an operator
on the improvement of the decay of the associated semigroup
was made in~\cite{KZ}, where an analog of Theorem~\ref{Thm.rate}
was proved for the decay rate in three-dimensional Dirichlet tubes.
We also refer to~\cite{FKP} where the conjecture
(for not necessarily self-adjoint operators)
is analysed from the point of view of heat kernels
and its relationship with Davies' conjecture~\cite{Davies_1997}
is observed.

The organization of this paper is as follows.
In the following Section~\ref{Sec.Pre}
we give a precise definition of the Laplacians $-\Delta_\theta$
and the associated semigroups.
The untwisted case is briefly treated in Section~\ref{Sec.straight},
obtaining, \emph{inter alia},
the first statement of Theorem~\ref{Thm.rate}.
The main body of the paper is represented by Section~\ref{Sec.ss}
where we develop the method of self-similar solutions
to get the improved decay rate of Theorem~\ref{Thm.rate}
(and also to establish
an alternative result, Theorem~\ref{Thm.rate.alt}).
The paper is concluded in Section~\ref{Sec.end}
by referring to physical interpretations of the result
and to some open problems.

\section{Preliminaries}\label{Sec.Pre}
%
The Laplacians $-\Delta_{\theta}$ are introduced
as the self-adjoint operators associated on $\sii(\Omega)$
with the quadratic form $\psi \mapsto \|\nabla\psi\|^2$
having the domains
$$
  \mathfrak{D}_\theta(\Omega) :=
  \left\{ \psi\in H^1(\Omega) \ | \
  \psi = 0 \ \, \mbox{on} \ \, \Gamma_\theta^D
  \right\}
  \,.
$$
Here and in the sequel $\|\cdot\|$~denotes
the norm in the Hilbert space~$\sii(\Omega)$.
It is possible to specify the operator domains (\cf~\cite{DKriz1}),
but we will not need them.
We only mention the result from~\cite{DKriz1} that the set
of restrictions of functions from $C_0^\infty(\Real^2)$ to~$\Omega$
that vanish on $\Gamma_\theta^D$ is dense in $\mathfrak{D}_\theta$
with respect to the $H^1(\Omega)$ norm (\cf~\cite[App.~B]{DKriz1}).

In view of~\eqref{spectrum}, both the operators $-\Delta_{\theta}$
satisfy the Poincar\'e-type inequality
\begin{equation}\label{Poincare}
  -\Delta_{\theta} \geq E_1
\end{equation}
in the sense of forms on $\sii(\Omega)$.
Here~$E_1$ is the first eigenvalue of
the one-dimensional operator $-\Delta_{DN}^{(-a,a)}$,
\ie\ the Laplacian in $\sii((-a,a))$
subject to the Dirichlet boundary condition at~$-a$
and Neumann boundary condition at~$a$.
This inequality is sharp for $-\Delta_{0}$,
while it follows from~\cite{KK2} that~\eqref{Poincare}
can be improved to a Hardy-type inequality~\eqref{Hardy} if $\theta=\pi$.

Recalling~\eqref{spectral.map} and that we are interested
in additional decay properties of~\eqref{heat.intro},
it is natural to rather consider the shifted parabolic problem
(obtained from the standard heat equation~\eqref{heat.intro}
by the replacement
$
  u(x,t) \mapsto e^{-E_1 t} \, u(x,t)
$):
\begin{equation}\label{heat}
  u_t - \Delta u - E_1 u = 0
  \,,
\end{equation}
subject to the initial condition~\eqref{heat.initial}
and the boundary conditions
\begin{equation}\label{heat.bc}
  u=0
  \quad\mbox{on}\quad \Gamma_\theta^D \times (0,\infty) \,,
  \qquad
  \frac{\partial u}{\partial n} =0
  \quad\mbox{on}\quad \Gamma_\theta^N \times (0,\infty) \,,
\end{equation}
where~$n$ denotes the normal vector to the boundary~$\partial\Omega$.

As usual, we consider the weak formulation of the problem
and, with an abuse of notation, we denote by the same symbol~$u$
both the function on $\Omega\times(0,\infty)$
and the mapping $(0,\infty) \to \sii(\Omega)$.
Standard semigroup theory implies that there exists
a unique solution of~\eqref{heat}--\eqref{heat.bc},
subject to the initial condition~\eqref{heat.initial},
that belongs to $C^0\big([0,\infty),\sii(\Omega)\big)$.
More precisely, the solution is given by $u(t) = S_\theta(t) u_0$,
where
\begin{equation}\label{semigroup}
  S_\theta(t) := e^{(\Delta_\theta+E_1)t}
\end{equation}
is the heat semigroup associated with
the shifted Laplacian $-\Delta_\theta-E_1$.

It is easy to see that the real and imaginary parts
of the solution~$u$ of~\eqref{heat} evolve separately.
By writing $u = \Re(u) + i \, \Im(u)$ and solving~\eqref{heat}
with initial data $\Re(u_0)$ and $\Im(u_0)$,
we may therefore reduce the problem to the case of a real function~$u_0$,
without restriction.
Consequently, all the functional spaces are considered to be real in the sequel.

\section{The untwisted strip}\label{Sec.straight}
%
If the strip is untwisted (\ie~$\theta=0$),
the heat equation~\eqref{heat}
can be easily solved by separation of variables.
Indeed, the Laplacian~$-\Delta_0$
can be identified with the decomposed operator
\begin{equation}\label{decomposition}
  (-\Delta^\Real) \otimes 1 + 1 \otimes
  (-\Delta_{DN}^{(-a,a)})
  \qquad\mbox{in}\qquad
  \sii(\Real) \otimes \sii((-a,a))
  \,,
\end{equation}
where~$-\Delta^\Real$ denotes the one-dimensional free Hamiltonian
(\ie~the usual self-adjoint realization of the Laplacian in~$\sii(\Real)$)
and~$1$ stands for the identity operators in the appropriate spaces.

The eigenvalues and (normalized) eigenfunctions of~$-\Delta_{DN}^{(-a,a)}$
are respectively given by ($n = 1,2,\dots$)
\begin{equation}\label{spectrum.straight}
  E_n := (2n-1)^2 E_1
  \,, \qquad
  \mathcal{J}_{n}(y_2)
  := \sqrt{\frac{1}{a}} \, \sin\left[ E_n (y_2+a) \right]
  \,,
\end{equation}
while the spectral resolution of~$-\Delta^\Real$
is obtained by the Fourier transform.
Then it is easy to see that the heat semigroup~$S_0(t)$
is an integral operator with kernel
\begin{equation}\label{kernel.straight}
  s_0(x,x',t) := \sum_{n=1}^\infty e^{-(E_n-E_1)t} \,
  \mathcal{J}_n(x_2) \, p(x_1,x_1',t) \, \mathcal{J}_n(x_2')
  \,,
\end{equation}
where
$$
  p(x_1,x_1',t) := \frac{e^{-(x_1-x_1')^2/(4t)}}{\sqrt{4\pi t}}
$$
is the well known heat kernel of~$-\Delta^\Real$.

Using the explicit form of the heat kernel,
it is straightforward to establish the following bounds:
\begin{Proposition}\label{Prop.straight}
There exists a constant~$C$ such that for every $t \geq 1$,
$$
  C^{-1} \, t^{-1/4}
  \leq
  \|S_0(t)\|_{\sii(\Omega,K) \to \sii(\Omega)}
  \leq
  C \, t^{-1/4}
  \,.
$$
\end{Proposition}
\begin{proof}
To get the lower bound, we may restrict
to the class of initial data~\eqref{heat.initial} of the form
$u_0(x)=\varphi(x_1)\mathcal{J}_1(x_2)$ with $\varphi\in\sii(\Real,K_1)$,
$K_1(x_1):=e^{x_1^2/4}$.
Then it is easy to see from~\eqref{kernel.straight} that
$$
  \|S_0(t)\|_{\sii(\Omega,K) \to \sii(\Omega)}
  \geq
  \|P(t)\|_{\sii(\Real,K_1) \to \sii(\Real)}
  \,,
$$
where~$P(t)$ is the heat semigroup of~$-\Delta^\Real$
for which the lower bound with~$t^{-1/4}$ is well known
(or can be easily established by taking
$\varphi=K_1^{-\alpha}$ with any $\alpha>1/2$
and evaluating the integrals with the kernel~$p$ explicitly).
On the other hand, using the Schwarz inequality,
we have
\begin{align*}
  \|S_0(t) u_0\|^2
  &\leq \|u_0\|_K^2
  \int_{\Omega\times\Omega} s_0(x,x',t)^2 \, K(x')^{-1} \, dx \, dx'
  \\
  & = \|u_0\|_K^2 \
  \sum_{n=1}^{\infty} e^{-2(E_n-E_1)t}
  \int_{\Real\times\Real} p(x_1,x_1',t)^2 \, K_1(x_1')^{-1} \, dx_1 \, dx_1'
\end{align*}
for every $u_0 \in \sii(\Omega,K)$.
Here the sum can be estimated by a constant independent of $t \geq 1$
and the integral (computable explicitly)
is proportional to~$t^{-1/2}$.
\end{proof}
\begin{Remark}
It is clear from the proof that the bounds hold
in less restrictive weighted spaces.
Indeed, it is enough to have a corresponding result
for the one-dimensional heat semigroup~$P(t)$.
\end{Remark}

As a consequence of Proposition~\ref{Prop.straight}, we get:
\begin{Corollary}\label{Corol.straight}
We have $\gamma_0=1/4$.
\end{Corollary}
\begin{proof}
The lower bound of Proposition~\ref{Prop.straight} implies $\gamma_0 \leq 1/4$.
The opposite inequality follows from the upper bound and~\eqref{spectral.map}.
\end{proof}
%

\section{The self-similarity transformation}\label{Sec.ss}
%
Our method to study the asymptotic behaviour of
the heat equation~\eqref{heat} is to adapt
the technique of self-similar solutions
used in the case of the heat equation in the whole Euclidean space
by Escobedo and Kavian~\cite{Escobedo-Kavian_1987}
to the present problem.
Following~\cite{KZ}, devoted to the analysis of the heat kernel
in twisted tubes,
we perform the self-similarity transformation
in the first (longitudinal) space variable only,
while keeping the other (transverse) space variable unchanged.

\subsection{An equivalent time-dependent problem}
More precisely, we consider a unitary transformation~$U$
on $\sii(\Omega)$ which associates to every solution
$
  u \in \sii_\mathrm{loc}\big((0,\infty),dt;\sii(\Omega,dx)\big)
$
of~\eqref{heat}
a self-similar solution~$\tilde{u}=U u$
in a new $s$-time weighted space
$
  \sii_\mathrm{loc}\big((0,\infty),e^s ds;\sii(\Omega,dy)\big)
$
via
\begin{equation*}
  \tilde{u}(y_1,y_2,s)
  = e^{s/4} u(e^{s/2}y_1,y_2,e^s-1)
  \,.
\end{equation*}
The inverse change of variables is given by
$$
  u(x_1,x_2,t)
  = (t+1)^{-1/4} \, \tilde{u}\big((t+1)^{-1/2}x_1,x_2,\log(t+1)\big)
  \,.
$$
When evolution is posed in that context,
$y=(y_1,y_2)$ plays the role of space variable and~$s$ is the new time.

It is easy to check that, in the new variables,
the evolution is governed by
\begin{equation}\label{heat.similar}
  \tilde{u}_s
  - \mbox{$\frac{1}{2}$} \, y_1 \;\! \partial_1 \tilde{u}
  - \partial_1^2 \tilde{u}
  - e^s \, \partial_2^2 \tilde{u}
  - E_1 \, e^s \, \tilde{u}
  - \mbox{$\frac{1}{4}$} \, \tilde{u}
  = 0
\end{equation}
subject to the same initial and boundary
conditions as~$u$ in~\eqref{heat.initial} and~\eqref{heat.bc}, respectively.

\begin{Remark}
Note that~\eqref{heat.similar} is a parabolic equation with
$s$-time-dependent coefficients.
The same occurs and has been previously analyzed
in twisted three-dimensional tubes~\cite{KZ}
and for a convection-diffusion equation in the whole space
but with a variable diffusion coefficient~\cite{Duro-Zuazua_1999}.
A careful analysis of the behaviour of the underlying elliptic operators
as~$s$ tends to infinity leads to a sharp decay rate for its solutions.
\end{Remark}

Since~$U$ acts as a unitary transformation on $\sii(\Omega)$,
it preserves the space norm of solutions
of~\eqref{heat} and~\eqref{heat.similar}, \ie,
\begin{equation}\label{preserve}
  \|u(t)\|=\|\tilde{u}(s)\|
  \,.
\end{equation}
This means that we can analyse the asymptotic time behaviour
of the former by studying the latter.

However, the natural space to study the evolution~\eqref{heat.similar}
is not $\sii(\Omega)$ but rather the weighted space~\eqref{weight}.
Following the approach of~\cite{KZ}
based on a theorem of J.~L.~Lions~\cite[Thm.~X.9]{Brezis_FR}
about weak solutions of parabolic equations
with time-dependent coefficients,
it can be shown that~\eqref{heat.similar}
is well posed in the scale of Hilbert spaces
\begin{equation*}
  \mathfrak{D}_\theta(\Omega,K)
  \subset \sii(\Omega,K) \subset
  \mathfrak{D}_\theta(\Omega,K)^*
  \,,
\end{equation*}
with
$$
  \mathfrak{D}_\theta(\Omega,K)
  := \left\{ \tilde{u}\in H^1(\Omega,K) \ | \
  \tilde{u} = 0 \ \, \mbox{on} \ \, \Gamma_\theta^D
  \right\}
  ,
$$
where
$
  H^1(\Omega,K)
$
denotes the usual weighted Sobolev space.

\subsection{Reduction to a spectral problem}
Multiplying the equation~\eqref{heat.similar} by $\tilde{u} K$
and integrating by parts (precisely this means that we use $\tilde{u} K$
as a test function in a weak formulation of~\eqref{heat.similar}),
we arrive at the identity
\begin{equation}\label{formal}
  \frac{1}{2} \frac{d}{ds} \|\tilde{u}(s)\|_K^2
  = - J_\theta^s[\tilde{u}(s)]
  \,.
\end{equation}
Here~$\|\cdot\|_K$ denotes the norm in~\eqref{weight} and
\begin{equation*}
  J_\theta^s[\tilde{u}] :=
  \|\partial_1\tilde{u}\|_{K}^2
  + e^s \, \|\partial_2\tilde{u}\|_{K}^2
  - E_1 \, e^s \, \|\tilde{u}\|_{K}^2
  - \frac{1}{4} \, \|\tilde{u}\|_{K}^2
\end{equation*}
is a closed quadratic form with domain
$
  \Dom(J_\theta^s) :=
  \mathfrak{D}_\theta(\Omega,K)
$
(independent of~$s$).
It remains to analyse the coercivity of~$J_\theta^s$.

More precisely, as usual for energy estimates,
we replace the right hand side of~\eqref{formal}
by the spectral bound, valid for each fixed $s \in [0,\infty)$,
\begin{equation}\label{spectral.reduction}
  \forall \tilde{u} \in \Dom(J_\theta^s) \;\!, \qquad
  J_\theta^s[\tilde{u}]
  \geq \mu_\theta(s) \, \|\tilde{u}\|_{K}^2
  \,,
\end{equation}
where~$\mu_\theta(s)$ denotes the lowest point in the spectrum of
the self-adjoint operator~$T_\theta^s$
associated in~$\sii(\Omega,K)$ with~$J_\theta^s$.
Then~\eqref{formal} together with~\eqref{spectral.reduction} implies
the exponential bound
\begin{equation}\label{spectral.reduction.integral}
  \forall s \in [0,\infty) \;\!, \qquad
  \|\tilde{u}(s)\|_{K}
  \leq \|\tilde{u}_0\|_{K} \,
  e^{-\int_0^s \mu_\theta(r) dr}
  \,.
\end{equation}
In this way, the problem is reduced to a spectral analysis
of the family of operators $\{T_\theta^s\}_{s \geq 0}$.

\subsection{Study of the spectral problem}
In order to investigate the operator~$T_\theta^s$ in~$\sii(\Omega,K)$,
we first map it into a unitarily equivalent operator
$\hat{T}_\theta^s := \mathcal{U} T_\theta^s \mathcal{U}^{-1}$
in~$\sii(\Omega)$ via the unitary transform
$$
  \mathcal{U}\;\!\tilde{u} := K^{1/2}\,\tilde{u}
  \,.
$$
By definition, $\hat{T}_\theta^s$~is the self-adjoint operator
associated in $\sii(\Omega)$ with the quadratic form
$
  \hat{J}_\theta^s[v] := J_\theta^s[\mathcal{U}^{-1}v]
$,
$
  v \in \Dom(\hat{J}_\theta^s) := \mathcal{U}\,\Dom(J_\theta^s)
$.
A straightforward calculation yields
\begin{equation}\label{J0.form}
\begin{aligned}
  \hat{J}_\theta^s[v]
  &= \|\partial_1 v\|^2
  + \frac{1}{16} \, \|y_1 v\|^2
  + e^s \, \|\partial_2 v\|^2
  - E_1 \;\! e^s \, \|v\|^2
  \,,
  \\
  v \in \Dom(\hat{J}_\theta^s)
  &= \mathfrak{D}_\theta(\Omega)
  \cap \sii(\Omega, y_1^2 \, dy)
  \,.
\end{aligned}
\end{equation}
In particular, $\Dom(\hat{J}_\theta^s)$~is independent of~$s$.
In the distributional sense, we can write
\begin{equation}\label{Hamiltonian}
  \hat{T}_\theta^s =
  -\partial_1^2 + \frac{1}{16} \, y_1^2
  - e^s \;\! \partial_2^2
  - E_1 \;\! e^s
  \,.
\end{equation}

We observe that the `longitudinal part' of~$\hat{T}_\theta^s$
coincides with the quantum harmonic-oscillator Hamiltonian
\begin{equation}\label{HO}
  H := - \frac{d^2}{dy_1^2} + \frac{1}{16} \, y_1^2
  \qquad\mbox{in}\qquad
  \sii(\Real)
\end{equation}
(\ie\ the Friedrichs extension
of this operator initially defined on $C_0^\infty(\Real)$).
We recall the well known fact that the form domain
$$
  \Dom(H^{1/2}) =
  H^1(\Real) \cap \sii(\Real, y_1^2 \, dy_1)
$$
is compactly embedded in $\sii(\Real)$,
so that the spectrum of~$H$ is purely discrete.
In fact, the spectrum can be computed explicitly
(see any textbook on quantum mechanics,
\eg, \cite[Sec.~2.3]{Griffiths}):
\begin{equation}\label{HO.spec}
  \sigma(H) = \left\{
  \frac{1}{2} \left(n+\frac{1}{2}\right)
  \right\}_{n=0}^\infty
  \,.
\end{equation}

Using now the discreteness of spectra of~$H$ and $-\Delta_{DN}^{(-a,a)}$
together with the minimax principle,
one may easily conclude that also $\hat{T}_\theta^s$
(and therefore~$T_\theta^s$) is an operator
with compact resolvent for all $s\in[0,\infty)$.
In particular, $\mu_\theta(s)$~represents the lowest eigenvalue of~$T_\theta^s$.

\subsection{The asymptotic behaviour of the spectrum}
In order to study the decay rate via~\eqref{spectral.reduction.integral},
we need information about the limit of the eigenvalue~$\mu_\theta(s)$
as the time~$s$ tends to infinity.
Notice that the scaling of the transverse variable in~\eqref{Hamiltonian}
corresponds to considering the operator~$\hat{T}_\theta^1$
in the shrinking strip $\Real \times (-e^{-s/2}a,e^{-s/2}a)$.
This suggests that~$\hat{T}_\theta^s$ will converge, in a suitable sense,
to a one-dimensional operator of the type~\eqref{HO}.
We shall see that the difference between the twisted ($\theta=\pi$)
and untwisted case ($\theta=0$) consists in that the limit operator
for the former is subject to an extra Dirichlet boundary condition at $y_1=0$.

Thus, simultaneously to~$H$ introduced in~\eqref{HO},
let us therefore consider the self-adjoint operator~$H_D$ in~$\sii(\Real)$
whose quadratic form acts in the same way as that of~$H$
but has a smaller domain
$$
  \Dom(H_D^{1/2}) :=
  \big\{
  \varphi\in\Dom(H^{1/2})\ |\ \varphi(0)=0
  \big\}
  \,.
$$

In fact, it is readily seen that~$\hat{T}_0^s$
can be identified with the decomposed operator
\begin{equation}\label{s.decomposition}
  H \otimes 1 + 1 \otimes
  (- e^s \;\! \Delta_{DN}^{(-a,a)} - E_1 \;\! e^s)
  \qquad\mbox{in}\qquad
  \sii(\Real) \otimes \sii((-a,a))
  \,,
\end{equation}
where~$1$ denotes the identity operators in the appropriate spaces.
Using~\eqref{HO.spec}, it follows that $\mu_0(s) = 1/4$
for all $s \in [0,\infty)$. Consequently,
\begin{equation}\label{untwisted.infinity}
  \mu_0(\infty) := \lim_{s\to\infty} \mu_0(s) = 1/4 \,.
\end{equation}
Moreover, \eqref{s.decomposition}~can be used to show
that~$\hat{T}_0^s$ converges to~$H$
in the norm-resolvent sense as $s \to \infty$,
if the latter is considered as an operator acting
on the subspace of~$\sii(\Omega)$ consisting of functions
of the form $\varphi(y_1)\mathcal{J}_{1}(y_2)$,
where~$\mathcal{J}_{1}$ is introduced in~\eqref{spectrum.straight}.

It is more difficult (and more interesting) to establish
the asymptotic behaviour of~$\mu_\pi(s)$.
A fine analysis of its behaviour leads to the key observation of the paper,
ensuring a gain of~$1/2$ in the decay rate in the twisted case.

We decompose the Hilbert space~$\sii(\Omega)$
into an orthogonal sum
\begin{equation}\label{direct}
  \sii(\Omega) = \mathfrak{H}_1 \oplus \mathfrak{H}_1^\bot
  \,,
\end{equation}
where the subspace~$\mathfrak{H}_1$ consists of functions
of the form
\begin{equation}\label{direct1}
  \psi_1(y) =
    \varphi(y_1)\,\mathcal{J}_1\big(\sgn(-y_1) \, y_2\big)
  \,.
\end{equation}
Notice that $y_2 \mapsto \mathcal{J}_{1}(-y_2)$
is an eigenfunction of $-\Delta_{ND}^{(-a,a)}$,
\ie\ the Laplacian in $\sii((-a,a))$
subject to the Neumann boundary condition at~$-a$
and Dirichlet boundary condition at~$a$
(reversed boundary conditions with respect to $-\Delta_{DN}^{(-a,a)}$).
Hence~$\psi_1$ satisfies the boundary conditions of~$-\Delta_\pi$.
Given any $\psi \in \sii(\Omega)$, we have the decomposition
$\psi = \psi_1 + \phi$ with $\psi_1\in\mathfrak{H}_1$ as above
and $\phi \in \mathfrak{H}_1^\bot$.
The mapping $\iota:\varphi\mapsto\psi_1$
is an isomorphism of $\sii(\Real)$ onto~$\mathfrak{H}_1$.
Hence, with an abuse of notations,
we may identify any operator~$h$ on $\sii(\Real)$
with the operator $\iota h \iota^{-1}$
acting on $\mathfrak{H}_1 \subset \sii(\Omega)$.
Having this convention in mind,
we state the following convergence result.
\begin{Proposition}\label{Prop.strong}
The operator $\hat{T}_\pi^s$ converges to
$
  H_D \oplus 0^\bot
$
in the strong-resolvent sense as $s \to \infty$, \ie,
\begin{equation*}
  \forall F \in \sii(\Omega) \,, \qquad
  \lim_{s \to \infty}
  \left\|
  \big(\hat{T}_\pi^s+1\big)^{-1} F
  - \left[\big(H_D + 1 \big)^{-1} \oplus 0^\bot\right] F
  \right\|
  = 0
  \,.
\end{equation*}
\end{Proposition}
\begin{proof}
We proceed as in the proof of~\cite[Prop.~5.4]{KZ}.
For any fixed $F \in \sii(\Omega)$,
let us set $\psi_s := (\hat{T}_\pi^s+1)^{-1}F$.
In other words, $\psi_s$~satisfies the resolvent equation
\begin{equation}\label{re}
  \forall v \in \Dom(\hat{J}_\pi^s) \,, \qquad
  \hat{J}_\pi^s(v,\psi_s) + (v,\psi_s)
  = (v,F)
  \,,
\end{equation}
where~$(\cdot,\cdot)$ denotes the inner product in~$\sii(\Omega)$
and $\hat{J}_\pi^s(\cdot,\cdot)$ is the sesquilinear form
associated with~\eqref{J0.form}.
In particular, choosing~$\psi_s$ for the test function~$v$ in~\eqref{re},
we have
\begin{multline}\label{resolvent.identity}
  \|\partial_1 \psi_s\|^2
  + \frac{1}{16} \, \|y_1 \psi_s\|^2
  + e^s \Big(
  \|\partial_2\psi_s\|^2 - E_1 \|\psi_s\|^2
  \Big)
  + \|\psi_s\|^2
  \\
  = (\psi_s,F)
  \leq \frac{1}{4} \, \|\psi_s\|^2 + \|F\|^2
  \,.
\end{multline}
Notice that $\|\partial_2\psi_s\|^2 \geq E_1 \|\psi_s\|^2$
by Fubini's theorem and the Poincar\'e inequality
for $-\Delta_{DN}^{(-a,a)}$ and $-\Delta_{ND}^{(-a,a)}$.
Consequently,
\begin{equation}\label{bounds}
  \|\partial_1 \psi_s\|^2 \leq C
  \,, \quad
  \|y_1 \psi_s\|^2 \leq C
  \,, \quad
  \|\psi_s\|^2 \leq C
  \,, \quad
  \|\partial_2\psi_s\|^2 - E_1 \|\psi_s\|^2
  \leq C e^{-s}
  \,,
\end{equation}
where~$C$ is a constant proportional to~$\|F\|^2$.

Now we employ the decomposition
$$
  \psi_s(y)
  = \varphi_s(y_1) \, \mathcal{J}_1\big(\sgn(-y_1) \, y_2\big)
  + \phi_s(y)
$$
where $\phi_s \in \mathfrak{H}_1^\bot$, \ie,
\begin{equation}\label{orthogonality}
  \forall y_1 \in \Real \,, \qquad
  \int_{-a}^a
  \mathcal{J}_1\big( \sgn(-y_1) \, y_2 \big) \,
  \phi_s(y_1,y_2) \, dy_2 = 0
  \,.
\end{equation}
That is, $y_2 \mapsto \phi_s(y_1,y_2)$ is orthogonal
to the ground-state eigenfunction of $-\Delta_{DN}^{(-a,a)}$
(respectively of $-\Delta_{ND}^{(-a,a)}$) if $y_1<0$
(respectively $y_1>0$).
Then
\begin{align*}
  \|\partial_2\psi_s\|^2 - E_1 \|\psi_s\|^2
  &=\|\partial_2\phi_s\|^2 - E_1 \|\phi_s\|^2
  \\
  &= \mbox{$\frac{1}{2}$} \;\! \|\partial_2\phi_s\|^2
  + \mbox{$\frac{1}{2}$} \;\! \|\partial_2\phi_s\|^2
  - E_1 \|\phi_s\|^2
  \\
  &\geq \mbox{$\frac{1}{2}$} \;\! \|\partial_2\phi_s\|^2
  + \big(\mbox{$\frac{1}{2}$}\;\!E_2-E_1\big) \;\! \|\phi_s\|^2
  \,,
\end{align*}
where $E_2=9E_1$ denotes the second eigenvalue
of~$-\Delta_{DN}^{(-a,a)}$
(which coincides with that of $-\Delta_{ND}^{(-a,a)}$).
Thus it follows from the last inequality of~\eqref{bounds} that
\begin{equation}\label{ri1}
  \|\phi_s\|^2 \leq C e^{-s}
  \qquad\mbox{and}\qquad
  \|\partial_2\phi_s\|^2 \leq C e^{-s}
  \,,
\end{equation}
where~$C$ is a constant proportional to~$\|F\|^2$.

It follows from~\eqref{bounds} that $\{\psi_s\}_{s>0}$
is a bounded family in $\Dom(\hat{J}_\pi^s)$.
Therefore it is precompact in the weak topology of $\Dom(\hat{J}_\pi^s)$.
Let~$\psi_\infty$ be a weak limit point,
\ie, for an increasing sequence of positive numbers $\{s_n\}_{n\in\Nat}$
such that $s_n \to \infty$ as $n \to \infty$,
$\{\psi_{s_n}\}_{n\in\Nat}$
converges weakly to~$\psi_\infty$ in $\Dom(\hat{J}_\pi^s)$.
Actually, we may assume that it converges strongly in $\sii(\Omega)$
because $\Dom(\hat{J}_\pi^s)$ is compactly embedded in $\sii(\Omega)$.
Since $\{\phi_{s_n}\}_{n\in\Nat}$ converges strongly to zero in $\sii(\Omega)$
due to~\eqref{ri1}, we know that $\psi_\infty \in \mathfrak{H}_1$, \ie,
$$
  \psi_\infty(y)
  = \varphi_\infty(y_1) \, \mathcal{J}_1\big(\sgn(-y_1) \, y_2\big)
$$
with some $\varphi_\infty \in \sii(\Real)$.
Since the weak derivative $\partial_1\psi_\infty \in \sii(\Omega)$ exists,
we necessarily have $\varphi_\infty \in H^1(\Real)$ and
$$
  \varphi_\infty(0) = 0
  \,.
$$

Finally, let $\varphi \in C_0^\infty(\Real\!\setminus\!\{0\})$ be arbitrary.
Taking
$$
  v(y) := \varphi(y_1) \, \mathcal{J}_1\big(\sgn(-y_1) \, y_2\big)
$$
as the test function in~\eqref{re}, with~$s$ being replaced by~$s_n$,
and sending~$n$ to infinity, we easily check that
\begin{equation*}
  (\dot\varphi,\dot\varphi_\infty)_{\sii(\Real)}
  + \frac{1}{16} \, (y_1\varphi,y_1\varphi_\infty)_{\sii(\Real)}
  + (\varphi,\varphi_\infty)_{\sii(\Real)}
  = (\varphi,f)_{\sii(\Real)}
  \,,
\end{equation*}
where
$$
  f(y_1) := \int_{-a}^a
  \mathcal{J}_1\big( \sgn(-y_1) \, y_2 \big) \,
  F(y_1,y_2) \, dy_2
  \,.
$$
That is, $\varphi_\infty = (H_D+1)^{-1} f$,
for \emph{any} weak limit point of $\{\varphi_s\}_{s \geq 0}$.
Summing up, we have shown that~$\psi_{s}$
converges strongly to $\psi_\infty$
in $\sii(\Omega)$ as $s \to \infty$,
where
$
  \psi_\infty(y)
  = \big[(H_D+z)^{-1} \oplus 0^\bot \big] F
$.
\end{proof}
\begin{Corollary}\label{Corol.strong}
One has
$$
  \mu_\pi(\infty) := \lim_{s\to\infty} \mu(s) = 3/4
  \,.
$$
\end{Corollary}
\begin{proof}
In general, the strong-resolvent convergence of Proposition~\ref{Prop.strong}
is not sufficient to guarantee the convergence of spectra.
However, in our case, since the spectra are purely discrete,
the eigenprojections converge even in norm (\cf~\cite{Weidmann_1980}).
In particular, $\mu_\pi(s)$~converges to the first eigenvalue of~$H_D$
as $s \to \infty$.
It remains to notice that the first eigenvalue of~$H_D$ coincides
(in view of the symmetry)
with the second eigenvalue of~$H$ which is~$3/4$ due to~\eqref{HO.spec}.
\end{proof}
%

\subsection{A lower bound to the decay rate}\label{Sec.improved}
%
We come back to~\eqref{spectral.reduction.integral}.
Recalling~\eqref{untwisted.infinity} and Corollary~\ref{Corol.strong},
we know that for arbitrarily small positive number~$\eps$
there exists a (large) positive time~$s_\eps$ such that
for all $s \geq s_\eps$, we have $\mu_\theta(s) \geq \mu_\theta(\infty) - \eps$.
Hence, fixing $\eps>0$, for all $s \geq s_\eps$, we have
\begin{align*}
  {-\int_0^s \mu_\theta(r) \, dr}
  &\leq {-\int_0^{s_\eps} \mu_\pi(r) \, dr} {-[\mu_\theta(\infty)-\eps](s-{s_\eps})}
  \\
  &\leq {[\mu_\theta(\infty)-\eps] s_\eps} {-[\mu_\theta(\infty)-\eps] s}
  \,,
\end{align*}
where the second inequality is due to the fact
that~$\mu_\theta(s)$ is non-negative for all $s \geq 0$
(it is in fact greater than or equal to~$1/4$,
\cf~Proposition~\ref{Prop.positivity} below).
At the same time, assuming $\eps \leq 1/4$, we trivially have
$$
  {-\int_0^s \mu_\theta(r) \, dr}
  \leq 0
  \leq {[\mu_\theta(\infty)-\eps] s_\eps} {-[\mu_\theta(\infty)-\eps] s}
$$
also for all $s \leq s_\eps$.
Summing up, for every $s \in [0,\infty)$, we have
\begin{equation}\label{instead}
  \|\tilde{u}(s)\|_{K}
  \leq C_\eps \, e^{-[\mu_\theta(\infty)-\eps]s} \, \|\tilde{u}_0\|_{K}
  \,,
\end{equation}
where $C_\eps := e^{s_\eps} \geq e^{[\mu_\theta(\infty)-\eps]s_\eps}$.

Now we return to the original variables
$
  (x_1,x_2,t) = (e^{s/2}y_1,y_2,e^s-1)
$.
Using~\eqref{preserve} together with the point-wise estimate $1 \leq K$,
and recalling that $\tilde{u}_0=u_0$,
it follows from~\eqref{instead} that
$$
  \|u(t)\|
  = \|\tilde{u}(s)\|
  \leq \|\tilde{u}(s)\|_{K}
  \leq C_\eps \, (1+t)^{-[\mu_\theta(\infty)-\eps]} \, \|u_0\|_{K}
$$
for every $t \in [0,\infty)$.
Consequently, we conclude with
$$
  \|S_\theta(t)\|_{\sii(\Omega,K) \to \sii(\Omega)}
  = \sup_{u_0 \in \sii(\Omega,K)\setminus\{0\}}
  \frac{\|u(t)\|}{\ \|u_0\|_{K}}
  \leq C_\eps \, (1+t)^{-[\mu_\theta(\infty)-\eps]}
$$
for every $t \in [0,\infty)$.
Since~$\eps$ can be made arbitrarily small,
this bound implies
$$
  \gamma_\theta \geq \mu_\theta(\infty)
  \,.
$$
This together with Corollary~\ref{Corol.straight}
proves Theorem~\ref{Thm.rate}.

\subsection{A global upper bound to the heat semigroup}\label{Sec.alt}
%
Theorem~\ref{Thm.rate} provides quite precise information
about the extra polynomial decay of solutions~$u$ of~\eqref{heat}
in a twisted tube in the sense that the decay rate~$\gamma_\pi$
is  better by a factor $1/2$ than in the untwisted case.
On the other hand, we have no control over
the constant~$C_\gamma$ in~\eqref{rate}
(in principle it may blow up as $\gamma \to \gamma_\theta$).
We therefore conclude this section by establishing
a global (in time) upper bound to the heat semigroup
(\ie~we get rid of the constant~$C_\gamma$)
but the prize we pay is just a qualitative knowledge
about the decay rate.
It is a consequence of~\eqref{spectral.reduction.integral}
and the following result:
\begin{Proposition}\label{Prop.positivity}
$
  \forall s \geq 0,
  \qquad
  \mu_0(s) = 1/4
  \,, \qquad
  \mu_\pi(s) > 1/4
  \,.
$
\end{Proposition}
\begin{proof}
The identity for~$\mu_0$ is readily seen from
the decomposition~\eqref{s.decomposition} and~\eqref{HO.spec}.
Using Fubini's theorem and the minimax principle,
it is also easy to deduce from~\eqref{J0.form} that
$\mu_\pi(s) \geq 1/4$ for all $s \geq 0$.
To show that the inequality is strict,
let us assume by contradiction that
$\mu_\pi(s) = 1/4$ for some $s \geq 0$.
Let~$v$ denote the corresponding eigenfunction of~$\hat{T}_\pi^s$.
Then the identity $\hat{J}_\pi^s[v]=\mu_\pi(s) \|v\|^2$ yields
\begin{equation}\label{identities}
  \|\partial_1 v\|^2
  + \frac{1}{16} \, \|y_1 v\|^2
  = \frac{1}{4} \, \|v\|^2
  \qquad\mbox{and}\qquad
  \|\partial_2 v\|^2 - E_1 \;\! \|v\|^2 = 0
  \,.
\end{equation}
Using the direct-sum decomposition~\eqref{direct},
the second identity implies that~$v$ is of the form~\eqref{direct1}.
The continuity of the eigenfunction~$v$ inside~$\Omega$
in turn requires that $\varphi(0)=0$.
However, this contradicts the first identity in~\eqref{identities}
which says, in view of~\eqref{HO.spec},
that~$\varphi$ is the first (therefore nowhere vanishing)
eigenfunction of~$H$.
\end{proof}

Combining this result with Corollary~\ref{Corol.strong},
we see that the number
\begin{equation}
  c_\theta := \inf_{s \in [0,\infty)}\mu_\theta(s) - 1/4
\end{equation}
is positive if $\theta=\pi$ and zero if $\theta=0$.
In any case, \eqref{spectral.reduction.integral}~implies
$$
  \|\tilde{u}(s)\|_{K}
  \leq \|\tilde{u}_0\|_{K} \, e^{-(c_\theta+1/4)s}
$$
for every $s \in [0,\infty)$.
Using this estimate instead of~\eqref{instead},
but following the same type of arguments as in Section~\ref{Sec.improved}
below~\eqref{instead}, we thus conclude with:
\begin{Theorem}\label{Thm.rate.alt}
We have
$$
  \forall t \in [0,\infty), \qquad
  \|S_\theta(t)\|_{\sii(\Omega,K) \to \sii(\Omega)}
  \leq (1+t)^{-(c_\theta+1/4)}
  \,,
$$
where $c_\pi > 0$ (and $c_0=0$).
\end{Theorem}
%

\section{Conclusions}\label{Sec.end}
%
The classical interpretation of the heat equation~\eqref{heat.intro}
is that its solution~$u$ gives the evolution of the temperature distribution
of a medium in the strip~$\Omega$ surrounded by a perfect insulator
on the Neumann boundary~$\Gamma_\theta^N$
and by a substance of a high-heat capacity and of zero temperature
on the Dirichlet boundary~$\Gamma_\theta^D$.
It also represents the simplest version of the stochastic Fokker-Planck
equation describing the Brownian motion in~$\Omega$
which is normally reflected on~$\Gamma_\theta^N$
and killed on~$\Gamma_\theta^D$
(\cf~\cite{Pinsky_2000} for a probabilistic setting in an analogous
higher-dimensional model).
Then the results of the present paper can be interpreted
as that the twisting of boundary conditions (\ie~$\theta=\pi$)
implies a faster cool-down/death
of the medium/Brownian particle in the strip.
Many other diffusive processes in nature are governed by~\eqref{heat.intro}.

Our proof that there is an extra decay rate for solutions of~\eqref{heat.intro}
if the boundary conditions are twisted was far from being straightforward.
This is a bit surprising because the result is quite
expectable from the physical interpretation,
if one notices that the twist makes it more difficult
for the Brownian particle to pass through the channel at $\{x_1=0\}$,
because of the proximity of killing boundary conditions.
At the same time, the Hardy inequality~\eqref{Hardy}
did not play any role in the proof of Theorem~\ref{Thm.rate}
(although, combining the theorem with the results of~\cite{KK2},
we eventually know that the existence of the Hardy inequality
is equivalent to the extra decay rate for the heat semigroup).
It would be desirable to find a more direct proof
of Theorem~\ref{Thm.rate} based on~\eqref{Hardy}.

We conjecture that the inequality of Theorem~\ref{Thm.rate}
can be replaced by equality, \ie, $\gamma_\pi=3/4$ for twisted strip.
The question of optimal value of the constant~$c_\pi$
(and its quantitative dependence on the half-width~$a$)
from Theorem~\ref{Thm.rate.alt} also constitutes an interesting open problem.
Note that the two quantities are related by
$c_\pi+1/4 \leq \gamma_\theta$.

The present paper can be viewed as a continuation
of the research initiated by our work~\cite{KZ},
where we investigated the large-time behaviour of
the heat semigroup in geometrically twisted Dirichlet tubes.
It confirms that the effect of twisting
(leading to the subcriticality of the Laplacian
and implying an improvement of the decay rate
of the associated heat semigroup)
is more general, namely it holds true
also in waveguide systems twisted via boundary conditions.
We expect that the extra decay rate will be induced
also in the systems twisted via embedding of the strip
into a negatively curved manifold,
for which the existence of Hardy inequalities is already known~\cite{K3}.

More generally, recall that we expect that there is always
an improvement of the decay rate for the heat semigroup
of a subcritical operator
(\cf~\cite[Conjecture in Sec.~6]{KZ} and~\cite[Conjecture~1]{FKP}).

\section*{Acknowledgment}
The work was partially supported by the Czech Ministry of Education,
Youth and Sports within the project LC06002,
and by Grant MTM2008-03541  the MICINN (Spain),
project PI2010-04 of the Basque Government
and the ERC Advanced Grant FP7-246775 NUMERIWAVES.

%
\bibliography{bib}

\providecommand{\bysame}{\leavevmode\hbox to3em{\hrulefill}\thinspace}
\providecommand{\MR}{\relax\ifhmode\unskip\space\fi MR }
\providecommand{\MRhref}[2]{%
  \href{http://www.ams.org/mathscinet-getitem?mr=#1}{#2}
}
\providecommand{\href}[2]{#2}
\begin{thebibliography}{10}

\bibitem{Brezis_FR}
H.~Br{\'e}zis, \emph{{Analyse fonctionnelle: Th\'eorie et applications}},
  Dunod, 2002.

\bibitem{Davies_1997}
E.~B. Davies, \emph{Non-{G}aussian aspects of heat kernel behaviour}, J.~London
  Math. Soc. (2) \textbf{55} (1997), 105--125.

\bibitem{DKriz1}
J.~Dittrich and J.~K{\v{r}}{\'{\i}}{\v{z}}, \emph{Bound states in straight
  quantum waveguides with combined boundary condition}, J. Math. Phys.
  \textbf{43} (2002), 3892--3915.

\bibitem{Duro-Zuazua_1999}
G.~Duro and E.~Zuazua, \emph{Large time behavior for convection-diffusion
  equations in {$\mathbb{R}^N$} with asymptotically constant diffusion},
  Commun. in Partial Differential Equations \textbf{24} (1999), 1283--1340.

\bibitem{Escobedo-Kavian_1987}
M.~Escobedo and O.~Kavian, \emph{Variational problems related to self-similar
  solutions of the heat equation}, Nonlinear Anal.-Theor. \textbf{11} (1987),
  1103--1133.

\bibitem{FKP}
M.~Frass, D.~Krej\v{c}i\v{r}\'{\i}k, and Y.~Pinchover, \emph{On some strong
  ratio limit theorems for heat kernels}, Discrete Contin. Dynam. Systems A
  \textbf{28} (2010), 495--509.

\bibitem{Griffiths}
D.~J. Griffiths, \emph{Introduction to quantum mechanics}, Prentice Hall, Upper
  Saddle River, NJ, 1995.

\bibitem{KK2}
H.~Kova\v{r}{\'\i}k and D.~Krej\v{c}i\v{r}\'{\i}k, \emph{{A Hardy inequality in
  a twisted Dirichlet-Neumann waveguide}}, Math. Nachr. \textbf{281} (2008),
  1159--1168.

\bibitem{K3}
D.~Krej\v{c}i\v{r}\'{\i}k, \emph{Hardy inequalities in strips on ruled
  surfaces}, J. Inequal. Appl. \textbf{2006} (2006), Article ID 46409, 10
  pages.

\bibitem{KZ}
D.~Krej\v{c}i\v{r}\'{\i}k and E.~Zuazua, \emph{The {H}ardy inequality and the
  heat equation in twisted tubes}, J. Math. Pures Appl (2010), to appear.

\bibitem{Pinsky_2000}
R.~G. Pinsky, \emph{A probabilistic approach to positive harmonic functions in
  a slab with alternating {D}irichlet and {N}eumann boundary conditions},
  Trans. Amer. Math. Soc. \textbf{352} (2000), 2445--2477.

\bibitem{Weidmann_1980}
J.~Weidmann, \emph{Continuity of the eigenvalues of self-adjoint operators with
  respect to the strong operator topology}, Integral Equations and Operator
  Theory \textbf{3} (1980), no.~1, 138--142.

\end{thebibliography}
\bibliographystyle{amsplain}
\end{document}